\documentclass[10pt,english,reqno]{amsart}

\usepackage{amsthm,amsmath,amssymb,amsfonts}
\usepackage{mathtools,mleftright}
\usepackage{booktabs,float,caption,setspace}
\usepackage{enumitem,cite}
\usepackage{epigraph}

\usepackage[
  pdfauthor={},
  pdfkeywords={},
  pdftitle={},
  pdfcreator={},
  pdfproducer={},
  linktocpage,colorlinks,bookmarksnumbered,linkcolor=blue,
  citecolor=red,urlcolor=red]{hyperref}

\theoremstyle{plain}
\newtheorem{theorem}{Theorem}
\newtheorem{proposition}[theorem]{Proposition}
\newtheorem{lemma}[theorem]{Lemma}
\newtheorem{corollary}[theorem]{Corollary}

\theoremstyle{definition}

\newtheorem{result}[theorem]{Result}
\newtheorem*{remark}{Remark}
\newtheorem{conjecture}[theorem]{Conjecture}
\newtheorem*{question}{Question}

\newcommand{\Case}[1]{\mbox{\texttt{Case} #1.}}

\setlist[enumerate]{label=(\roman*),font=\rm,leftmargin=1.2cm,itemsep=1pt,
parsep=1pt,before={\parskip=1pt}}
\newcommand{\listeq}[1]{\hspace*{-1.2cm}\hfill$\displaystyle #1$\hfill\hspace{1sp}}
\numberwithin{theorem}{section}
\numberwithin{equation}{section}
\numberwithin{table}{section}
\numberwithin{figure}{section}

\newcommand{\ZZ}{\mathbb{Z}}
\newcommand{\QQ}{\mathbb{Q}}

\newcommand{\CC}{\mathbb{C}}
\newcommand{\FF}{\mathbb{F}}

\DeclareMathOperator{\denom}{denom}
\DeclareMathOperator{\ord}{ord}
\DeclareMathOperator{\lcm}{lcm}
\DeclareMathOperator{\re}{Re}

\newcommand{\BF}{\mathrm{B}}
\newcommand{\MS}{\mathcal{S}}
\newcommand{\MSR}{\mathcal{S}^\star}
\newcommand{\MP}{\mathcal{P}}
\newcommand{\LS}{\mathcal{L}}
\newcommand{\NI}{\mathcal{N}}

\newcommand{\dmid}{\parallel}
\newcommand{\mids}{\, \mid \,}

\newcommand{\impliesq}{\; \implies \;}
\newcommand{\andq}{\quad \text{and} \quad}

\newcommand{\set}[1]{\mleft\{#1\mright\}}

\begin{document}

\title[On the nonintegrality of certain generalized binomial sums]
{On the nonintegrality of certain\\ generalized binomial sums}
\author{Bernd C. Kellner}
\address{G\"ottingen, Germany}
\email{bk@bernoulli.org}

\begin{abstract}
We consider certain generalized binomial sums $\mathcal{S}_{(r,n)}(\ell)$ and
discuss the nonintegrality of their values for integral parameters $n,r \geq 1$
and $\ell \in \mathbb{Z}$ in several cases using $p$-adic methods.
In particular, we show some properties of the denominator of $\mathcal{S}_{(r,n)}(\ell)$.
Viewed as polynomials, the sequence $(\mathcal{S}_{(r,n)}(x))_{n \geq 0}$ forms
an Appell sequence. The special case $\mathcal{S}_{(r,n)}(2)$ reduces to the sum
$\sum_{k=0}^{n} \binom{n}{k} \frac{r}{r+k}$,
which has recently received some attention from several authors regarding the
conjectured nonintegrality of its values. So far, only a few cases have been proved.
The generalized results imply, among other things, for even $|\ell| \geq 2$ that
$\mathcal{S}_{(r,n)}(\ell) \notin \mathbb{Z}$ when $\binom{r+n}{r}$ is even,
e.g., $r$ and $n$ are odd. Although there exist exceptions where
$\mathcal{S}_{(r,n)}(\ell) \in \mathbb{Z}$, ``almost all'' values of
$\mathcal{S}_{(r,n)}(\ell)$ for $n,r \geq 1$ are nonintegral for any fixed
$|\ell| \geq 2$. Subsequently, we also derive explicit inequalities between the
parameters for which $\mathcal{S}_{(r,n)}(\ell) \notin \mathbb{Z}$.
Especially, this is shown for certain small values of $\ell$ for $r \geq n$ and
$n > r \geq \frac{1}{5} n$. As a supplement, we finally discuss exceptional cases
where $\mathcal{S}_{(r,n)}(\ell) \in \mathbb{Z}$.
\end{abstract}

\maketitle

\setlength{\epigraphwidth}{13.5em}
\epigraph{\footnotesize \it If you cannot solve a problem,
then try to solve a more general problem.}
{\footnotesize P\'olya~\cite{Polya:2004}}
\vspace{4ex}


\section{Introduction}

Define the monic polynomial
\begin{equation} \label{eq:S-def}
  \MS_{(r,n)}(x) = \sum_{k=0}^{n} \binom{n}{k} (-1)^k \, x^{n-k}
    \binom{r+k}{r}^{\!\!-1} \in \QQ[x]
\end{equation}
of degree $n$ for integers $n, r \geq 0$. Trivial cases are given by
\begin{equation} \label{eq:S-0}
  \MS_{(0,n)}(x) = (x-1)^n \andq \MS_{(r,0)}(x) = 1.
\end{equation}
Therefore, we assume that $n, r \geq 1$ for the rest of the paper.

As we shall see later, the polynomial \eqref{eq:S-def} can be expressed in
several different ways that lead to various properties.
As a surprising relation, we have
\begin{equation} \label{eq:S-2}
  \MS_{(r,n)}(2) = \sum_{k=0}^{n} \binom{n}{k} \frac{r}{r+k}.
\end{equation}
The above sum has received some attention in recent times, where it is
conjectured that \eqref{eq:S-2} only takes nonintegral values. This has been
shown for $1 \leq r \leq 22$ and for $1 \leq n < r$ in that case.
See \cite{Chirita:2014,LLPT:2020,Lopez-Aguayo:2015,Lopez-Aguayo&Luca:2016,
Luca&Pomerance:2021,Thongjunthug:2019}
for the history and results.

Interestingly, the following sum, related to \eqref{eq:S-2} with alternating signs,
\[
  \MS_{(r,n)}(0) = \sum_{k=0}^{n} \binom{n}{k} (-1)^{n-k} \frac{r}{r+k}
    \notin \ZZ,
\]
can be evaluated instantly, since it can be interpreted as
a finite difference as well as a partial fraction decomposition
(see Corollary~\ref{cor:basic} and Section~\ref{sec:proofs-1}).

However, a generalized conjecture of \eqref{eq:S-2} cannot be established
without further study, since there are several exceptions where in fact
$\MS_{(r,n)}(\ell) \in \ZZ$ for certain $\ell \in \ZZ$ as listed in the
two tables below. See Section~\ref{sec:except} for more results.

\begin{table}[H]
\setstretch{1.25}
\begin{center}
\begin{tabular}{*{6}{c}}
  \toprule
  \multicolumn{6}{c}{Parameters $(r,n,\ell)$} \\
  \midrule
  $(2,4,-49)$ & $(2,4,-34)$ & $(2,4,-19)$ & $(2, 4, -4)$ & $(2, 4, 11)$ & $(2, 4, 26)$ \\
  $(2, 4, 41)$ & $(2, 8, -17)$ & $(2, 8, 28)$ & $(2, 12, -38)$ & $(2, 16, -16)$ & $(2, 16, 35)$ \\
  $(2, 20, -5)$ & $(2, 40, -40)$ & $(3, 7, -17)$ & $(3, 7, 43)$ & $(4, 6, 43)$ & $(4, 24, 46)$ \\
  \bottomrule
\end{tabular}

\caption{Exceptions where $\MS_{(r,n)}(\ell) \in \ZZ$ in the range
$1 \leq |\ell|,n,r \leq 50$.}
\label{tbl:except-1}
\end{center}
\end{table}
\vspace*{-8pt}

One observes that the exceptions in Table~\ref{tbl:except-1} have the property
that $r < n$ and $\ell$ is relatively small. In contrast, Table~\ref{tbl:except-2}
shows exceptions of the opposite case $n \leq r$,
which reveals that the least positive $\ell$ can be arbitrarily large.

\begin{table}[H]
\setstretch{1.25}
\begin{center}
\begin{tabular}{*{5}{c}}
  \toprule
  \multicolumn{5}{c}{Parameters $(r,n,\ell)$} \\
  \midrule
  $(9,6,1002)$ & $(10,5,2003)$ & $(12,8,50\,389)$ & $(12,9,41\,991)$ & $(16, 12, 4\,345\,966)$ \\
  \bottomrule
\end{tabular}

\caption{\parbox[t]{23em}{Exceptions where $\MS_{(r,n)}(\ell) \in \ZZ$ with
$1 \leq n \leq r \leq 16$ and least positive $\ell$.}}
\label{tbl:except-2}
\end{center}
\end{table}
\vspace*{-8pt}

The purpose of the paper is to discuss the phenomenon of the nonintegrality of
the sum $\MS_{(r,n)}(\ell)$ in spite of exceptions and to derive explicit
conditions for its parameters.
Indeed, the motivation for the generalized results was induced by the above
quotation of P\'olya, since the sum~\eqref{eq:S-2} sheds no light on its
behavior when viewed individually.

The paper is organized as follows. The next section presents some basic
properties of the polynomial $\MS_{(r,n)}(x)$ and its values, while
Section~\ref{sec:main} contains the main results.
Subsequently, Section~\ref{sec:prelim} is devoted to preliminaries and some known
results in number theory. Sections~\ref{sec:proofs-1} and~\ref{sec:proofs-2}
contain the proofs of the theorems. The last section discusses the case of exceptions.


\section{Basic properties}

Let $(n)_k$ denote the falling factorial such that $\binom{n}{k} = (n)_k / k!$.
Let $\denom(\cdot)$ be the denominator of a rational polynomial or number.
For properties of Appell polynomials, see \cite{Appell:1880,Roman:1984,Rota:1975}.
The following theorem shows some basic properties of $\MS_{(r,n)}(x)$.

\begin{theorem} \label{thm:basic}
Let $n, r \geq 1$. There are the following identities:
\begin{enumerate}
\item \listeq{ \hspace*{-0.2em}
  \MS_{(r,n)}(x) = \sum_{k=0}^{n} \frac{(n)_k}{(r+k)_k} \, (-1)^k \, x^{n-k};}

\item \listeq{
  \MS_{(r,n)}(x) = r\! \int_{0}^{1} (x-t)^n (1-t)^{r-1} dt.}
\end{enumerate}
The polynomial $\MS_{(r,n)}(x)$ is an Appell polynomial satisfying the
equivalent relations
\begin{enumerate}[resume]
\item \listeq{ \hspace*{-2.6em}
  \MS_{(r,n)}(x)' = n \, \MS_{(r,n-1)}(x);}

\item \listeq{
  \MS_{(r,n)}(x + y) = \sum_{k=0}^{n} \binom{n}{k} \MS_{(r,k)}(x) \, y^{n-k}.}
\end{enumerate}
The denominator of $\MS_{(r,n)}(x)$ and its values for $\ell \in \ZZ$ have
the properties
\begin{enumerate}[resume]
\item \listeq{
  \denom\left( \MS_{(r,n)}(x) \right) = \binom{r+n}{r};}

\item \listeq{
  \denom\left( \MS_{(r,n)}(\ell) \right) \,\,\mid\,\, \binom{r+n}{r}.}
\end{enumerate}
\end{theorem}

Evaluating the integral formula and using the Appell properties of
$\MS_{(r,n)}(x)$ easily imply the following results.

\begin{corollary} \label{cor:basic}
Let $n, r \geq 1$. We have
\begin{enumerate}
\item \listeq{
  \MS_{(r,n)}(x+1) = \sum_{k=0}^{n} \binom{n}{k} x^{n-k} \frac{r}{r+k}.}
\end{enumerate}
Special values are given as follows:
\begin{enumerate}[resume]
\item \listeq{
  \MS_{(r,n)}(0) = \sum_{k=0}^{n} \binom{n}{k} (-1)^{n-k} \frac{r}{r+k}
    = (-1)^n \, \binom{r+n}{r}^{\!\!-1} \notin \ZZ;}

\item \listeq{ \hspace*{-1.95em}
  \MS_{(r,n)}(1) = \frac{r}{r+n} \notin \ZZ;}

\item \listeq{ \hspace*{-0.25em}
  \MS_{(r,n)}(2) = \sum_{k=0}^{n} \binom{n}{k} \frac{r}{r+k}.}
\end{enumerate}
Moreover, the values of $\MS_{(r,n)}(x)$ have the properties
\begin{enumerate}[resume]
\item \listeq{
  \MS_{(r,n)}(x) > 0 \quad (x \geq 1) \andq
  (-1)^n \, \MS_{(r,n)}(x) > 0 \quad (x \leq 0).}
\end{enumerate}
\end{corollary}

It turns out that $\MS_{(r,n)}(-1)$ is related to partial sums of binomial
coefficients in a row of Pascal's triangle. So far, no closed forms are known
for such sums according to \cite[Sec.~5.1, pp.~165--167]{GKP:1994}.

\begin{theorem} \label{thm:row}
Let $n, r \geq 1$. There are the following identities:
\begin{enumerate}
\item \listeq{
  \MS_{(r,n)}(x) = \binom{r+n}{r}^{\!\!-1}
    \sum_{k=0}^{n} \binom{r+n}{k} (-1)^{n-k} \, x^k;}

\item \listeq{ \hspace*{-2.9em}
  \MS_{(r,n)}(-1) = (-1)^n \binom{r+n}{r}^{\!\!-1} \sum_{k=0}^{n} \binom{r+n}{k}.}

\item We have the reciprocity relation
\[
  (-1)^n \MS_{(r,n)}(-1) + (-1)^r \MS_{(n,r)}(-1)
    = 2^{r+n} \binom{r+n}{r}^{\!\!-1} + 1.
\]

\item At least one of the values of $\set{\MS_{(r,n)}(-1), \MS_{(n,r)}(-1)}$
is not in $\ZZ$. In particular,
\[
  \MS_{(n,n)}(-1),\, \MS_{(n+1,n)}(-1),\, \MS_{(n,n+1)}(-1) \notin \ZZ.
\]
\end{enumerate}
\end{theorem}

In contrast, the related sum to $\MS_{(r,n)}(-1)$ with alternating signs,
\[
  \MS_{(r,n)}(1) = \binom{r+n}{r}^{\!\!-1}
    \sum_{k=0}^{n} \binom{r+n}{k} (-1)^{n-k}
    = \frac{r}{r+n},
\]
is solvable at once by Corollary~\ref{cor:basic}~(iii).

\begin{corollary} \label{cor:row}
Let $n, r \geq 1$. We have
\begin{align*}
  \sum_{k=0}^{n} \binom{r+n}{k}
    &= r \binom{r+n}{r} \int_{0}^{1} (1+t)^n (1-t)^{r-1} dt \\
    &= r \binom{r+n}{r} \sum_{k=0}^{n} \binom{n}{k} (-1)^k \frac{2^{n-k}}{r+k}.
\end{align*}
\end{corollary}

The reciprocity relation of Theorem~\ref{thm:row} can be given in a generalized
form, which then has a different shape. Define the reciprocal polynomial
\[
  \MSR_{(r,n)}(x) = x^n \MS_{(r,n)}(x^{-1}).
\]

\begin{theorem} \label{thm:recip}
Let $n, r \geq 1$. We have the reciprocity relation
\[
  \MS_{(r,n)}(x) + x^n \MSR_{(n,r)}(x)
    = (-1)^r (x-1)^{r+n} \binom{r+n}{r}^{\!\!-1} + x^n.
\]
\end{theorem}

For the next applications, we need some recurrence formulas.

\begin{proposition} \label{prop:recur}
Let $n, r \geq 1$. There are the following recurrence formulas:
\begin{enumerate}
\item \listeq{ \hspace*{-0.2em}
  \MS_{(r,n)}(x) = (x-1) \, \MS_{(r,n-1)}(x) + \frac{r}{r+1} \, \MS_{(r+1,n-1)}(x);}

\item \listeq{ \hspace*{-6.9em}
  \MS_{(r,n)}(x) = x^n - \frac{n}{r+1} \, \MS_{(r+1,n-1)}(x);}

\item \listeq{ \hspace*{-3.1em}
  \MS_{(r+1,n)}(x) = \frac{r+1}{r+n+1} \left( x^{n+1} - (x-1) \, \MS_{(r,n)}(x) \right)\!;}

\item \listeq{
  \MS_{(r,n+1)}(x) = \frac{r}{r+n+1} x^{n+1} + \frac{n+1}{r+n+1} (x-1) \, \MS_{(r,n)}(x).}
\end{enumerate}
\end{proposition}

To tackle the problem of the nonintegrality and to obtain divisibility properties,
it is convenient to find a further representation of $\MS_{(r,n)}(x)$ as follows.

\begin{theorem} \label{thm:ident}
Let $n, r \geq 1$. We have
\begin{align}
  \MS_{(r,n)}(x) &= (-1)^{r} r! \, \frac{(x-1)^{r+n} - x^{n+1} \, \psi_{(r,n)}(x)}
    {(n+1) \dotsm (n+r)},
  \label{eq:S-ident} \\
\shortintertext{where}
  \psi_{(r,n)}(x) &= \sum_{k=0}^{r-1} \binom{n+k}{k} (-1)^k (x-1)^{r-1-k}
  \label{eq:psi-1} \\
  &= \sum_{k=0}^{r-1} \binom{n+r}{k} (-1)^k x^{r-1-k}.
  \label{eq:psi-2}
\end{align}

\newpage\noindent
In particular, there are the special cases:
\begin{enumerate}
\item \listeq{ \hspace*{-7em}
  \MS_{(1,n)}(x) = \frac{x^{n+1}-(x-1)^{n+1}}{n+1};}

\item \listeq{ \hspace*{-10.7em}
  \MS_{(r,1)}(x) = x - \frac{1}{r+1};}

\item \listeq{
  \MS_{(r,n)}(2) = (-1)^{r} r! \, \frac{1 - 2^{n+1}
    \sum_{k=0}^{r-1} (-1)^k \binom{n+k}{k}}{(n+1) \cdots (n+r)}.}
\end{enumerate}
\end{theorem}


\section{Main results}
\label{sec:main}

In this section, we derive several conditions on the nonintegrality of
$\MS_{(r,n)}(\ell)$. Let $p$ denote always a prime.
Let $\ord_p(n)$ and $s_p(n)$ be the $p$-adic valuation and the sum of base-$p$
digits of $n$, respectively. The notation ${p^e \dmid n}$ means that
${p^e \mid n}$ but ${p^{e+1} \nmid n}$, i.e., $\ord_p(n) = e$.
The following two results give conditions to test the (non-) integrality via
congruences.

\begin{proposition} \label{prop:congr-1}
Let $n, r \geq 1$ and $\ell \in \ZZ$. Then $\MS_{(r,n)}(\ell) \in \ZZ$
if and only if
\[
  \sum_{k=0}^{n-1} \binom{r+n}{k} (-\ell)^k \equiv 0 \pmod{\binom{r+n}{r}}.
\]
\end{proposition}

\begin{proposition} \label{prop:congr-2}
Let $n, r \geq 1$ and $\ell \in \ZZ \setminus \set{0,1}$.
If there exists an index $d \in \set{1,\ldots,r}$ where
\[
  r! \, (\ell-1)^{r-d} \left( (\ell-1)^{n+d} - \ell^{n+d} \right)
    \not\equiv 0 \pmod{n+d},
\]
then $\MS_{(r,n)}(\ell) \notin \ZZ$.
\end{proposition}

Regarding the properties of $\MS_{(r,n)}(\ell)$,
we have a kind of reciprocity relation between the parameters $n$ and $r$,
as well as a symmetry relation of $\ell$.

\begin{theorem} \label{thm:main}
Let $n, r \geq 1$ and $\ell \in \ZZ$.
Set $g = \gcd \mleft( \binom{r+n}{r}, \ell \mright)$
and $e_p = \ord_p \mleft( \binom{r+n}{r} \mright)$.
Assume that one of the following conditions holds:
\begin{enumerate}
\item $r=1$ or $n=1$;
\item $r=n$;
\item $r+n$ is a prime power;
\item $r$ and $n$ are odd, and $\ell$ is even;
\item $\ell \in \set{-1,0,1}$;
\item $g \neq 1$.
\end{enumerate}
Then we have that $\MS_{(r,n)}(\pm \ell) \notin \ZZ$ and $\MS_{(n,r)}(\pm \ell) \notin \ZZ$,
except for the case when only condition~(v) holds with $\ell = \pm 1$,
where at least $\MS_{(r,n)}(-1) \notin \ZZ$ or $\MS_{(n,r)}(-1) \notin \ZZ$.
Moreover, if~$g \neq 1$, then we have for each prime divisor $p \mid g$ that
\[
  p^{e_p} \dmid \denom( \MS_{(r,n)}(\pm \ell) ) \andq
  p^{e_p} \dmid \denom( \MS_{(n,r)}(\pm \ell) ).
\]
\end{theorem}

The diagonal case $r=n$ can be handled in more detail as follows.

\begin{theorem} \label{thm:diag}
Let $n \geq 1$. We have
\[
  \MS_{(n,n)}(x) = \binom{2n}{n}^{\!\!-1}
    \sum_{k=0}^{n} \binom{2n}{k} (-1)^{n-k} \, x^k,
\]
which obeys the recurrence
\[
  2\MS_{(n+1,n+1)}(x)
    = \frac{n+1}{2n+1} (x-1) \left( x^{n+1}
    - (x-1) \MS_{(n,n)}(x)\right) + x^{n+1}
\]
with $\MS_{(1,1)}(x) = x - \frac{1}{2}$.
For $\ell \in \ZZ$, we have $\MS_{(n,n)}(\ell) \notin \ZZ$.
More precisely,
\[
  \ord_2(\denom(\MS_{(n,n)}(\ell))) =
    \begin{cases}
      1, & \text{if $\ell$ is odd};\\
      s_2(n), & \text{if $\ell$ is even}.\\
    \end{cases}
\]
\end{theorem}

Since for any given positive integer $\ell$, almost all binomial coefficients
(in the sense of a density) are divisible by $\ell$ (this is due to Singmaster;
see Theorem~\ref{thm:Singmaster}), this implies the following corollary of
Theorem~\ref{thm:main}.

\begin{corollary} \label{cor:density}
Define the following sets for $m \geq 2$ and $\ell \in \ZZ$:
\[
  \NI_m(\ell) = \set{ (r,n) \in \ZZ^2 : n,r \geq 1, \, 1 \leq r+n \leq m,
    \;\text{and}\;\; \MS_{(r,n)}(\ell) \notin \ZZ}.
\]
If $|\ell| \geq 2$, then we have the density
\[
  \lim_{m \to \infty} \# \NI_m(\ell) \Big/ \binom{m}{2} = 1,
\]
which implies that almost all values of $\MS_{(r,n)}(\ell)$ are nonintegral for
$n,r \geq 1$.
\end{corollary}

Using Theorem~\ref{thm:main} for even $\ell \in \ZZ \setminus \set{0}$, we
arrive at a particular case. If $\binom{r+n}{r}$ is even (e.g., $r$ and $n$
are odd), then $\MS_{(r,n)}(\ell) \notin \ZZ$. This coincides with Pascal's
triangle modulo~$2$, which is known as the Sierpi\'{n}ski gasket~\cite{Sierpinski:1915}.
See Figure~\ref{fig:triangle}, where small black triangles represent
the odd binomial coefficients, and the blanks represent the even ones.

\begin{figure}[H]
\begin{center}
\includegraphics[width=5cm]{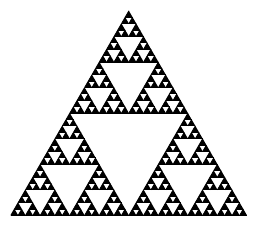}

\caption{Sierpi\'{n}ski gasket.}
\label{fig:triangle}
\end{center}
\end{figure}

At the end of this section, we consider the situation of inequalities between
the parameters of $\MS_{(r,n)}(\ell)$, supplementing the results of
Theorem~\ref{thm:main}. We use several known results on primes in short intervals,
which will be introduced in Section~\ref{sec:prelim}.

\begin{theorem} \label{thm:range}
Let $n, r \geq 2$ and $\ell \in \ZZ \setminus \set{0,1}$.
Set $g = \gcd \mleft( \binom{r+n}{r}, \ell-1 \mright)$ and
$\MP = \set{p : p > \frac{3}{2} n,\, p \mid \binom{r+n}{r},\, p \nmid \ell-1}$.
We have $\MS_{(r,n)}(\ell) \notin \ZZ$,
if one of the following mutually exclusive conditions holds:
\begin{enumerate}
\item $n > r \geq \frac{1}{5} n$ where $n \geq |\ell-1|$
(if $n \geq 89\,693$, then $r \geq \frac{1}{5} n$ can be improved by
$r > n / \!\log^3 n$);

\item $r > n$ where $n \geq \frac{2}{3} |\ell-1|$ or $g=1$ or $\MP \neq \emptyset$.
\end{enumerate}
\end{theorem}

The exceptions $(r,n,\ell) \in \set{(2, 4, -4),(2, 20, -5)}$ of
Table~\ref{tbl:except-1} show that the conditions of Theorem~\ref{thm:range}~(i),
namely, $n > r \geq \frac{1}{5} n$ and $n \geq |\ell-1|$, are essentially needed.
Regarding part~(ii), the exceptions of Table~\ref{tbl:except-2} imply that the
condition ${r > n}$ generally requires an additional condition on $\ell$.
The special case ${r > n \geq 1}$ for $\ell = 2$ was proved by
L\'opez-Aguayo and Luca~\cite{Lopez-Aguayo&Luca:2016}
for the sum~\eqref{eq:S-2}. For small values $\ell \in \LS$, where
\[
  \LS = \set{-3,-2,-1,2,3,4,5},
\]
we finally achieve the following result with simpler conditions.

\begin{corollary} \label{cor:range}
Let $n, r \geq 1$ and $\ell \in \LS$.
If $r \geq n$ or $n > r \geq \frac{1}{5} n$, then $\MS_{(r,n)}(\ell) \notin \ZZ$.
\end{corollary}

If one could remove the above restriction $r \geq \frac{1}{5} n$, then this would
prove the existing conjecture of the nonintegrality of $\MS_{(r,n)}(\ell)$ for
$\ell = 2$ as well as for other small values of $\ell$. We may raise the extended
conjecture as follows.

\begin{conjecture}
If $n, r \geq 1$ and $\ell \in \LS$, then $\MS_{(r,n)}(\ell) \notin \ZZ$.
\end{conjecture}

We conclude with the following question.

\begin{question}
For which numbers $\ell \in \ZZ \setminus \set{0,1}$ does
$\MS_{(r,n)}(\ell)$ take only nonintegral values for all $n, r \geq 1$?
\end{question}


\section{Preliminaries}
\label{sec:prelim}

Let $\ZZ_p$ be the ring of $p$-adic integers and $\QQ_p$ be the field of $p$-adic
numbers. Extend $\ord_p(s)$ as the $p$-adic valuation of $s \in \QQ_p$.
Let $\FF_p$ be the finite field with $p$ elements.

Applying Legendre's formula \cite[pp. 8--10]{Legendre:1808}
\[
  \ord_p( n! ) = \frac{n - s_p(n)}{p-1}
\]
to binomial coefficients provides that
\begin{equation} \label{eq:binom-2}
  \ord_p \left( \binom{n}{k} \right) = \frac{s_p(k) + s_p(n-k) - s_p(n)}{p-1}.
\end{equation}

\begin{lemma}[\hspace{1sp}{\cite[Sec.~5.1, p.~37]{Robert:2000}}] \label{lem:val-min}
If $n \geq 1$, then
\[
  \ord_p \mleft( \sum_{\nu = 0}^n x_\nu \mright) \geq
    \min_{0 \leq \nu \leq n} \ord_p( x_\nu ) \quad (x_\nu \in \QQ_p),
\]
where equality holds, if there exists an index $m$ such that
$\ord_p( x_m ) < \ord_p( x_\nu )$ for all $\nu \neq m$.
\end{lemma}

\begin{lemma} \label{lem:binom-p}
If $n = p^e$ with $p$ a prime and $e \geq 1$, then
\[
  \ord_p \mleft( \binom{n}{k} \mright) \geq 1 \quad (0 < k < n).
\]
\end{lemma}

\begin{proof}
This follows from applying the Frobenius endomorphism in $\FF_p$
iteratively such that
\[
  (x+y)^{p^\nu} = x^{p^\nu} + y^{p^\nu} \quad (\nu \geq 1). \qedhere
\]
\end{proof}

\begin{theorem}[Singmaster \cite{Singmaster:1974}] \label{thm:Singmaster}
Let $d, m \geq 1$. Define the sets
\[
  B_m(d) = \set{ (j,k) \in \ZZ^2 : j,k \geq 0, \; 0 \leq j+k < m, \,\text{and}\;\;
    d \mid \binom{j+k}{k} }.
\]
Then we have the density
\[
  \lim_{m \to \infty} \# B_m(d) \Big/ \binom{m+1}{2} = 1,
\]
which implies that almost all binomial coefficients are divisible by $d$.
\end{theorem}

\begin{theorem}[Faulkner \cite{Faulkner:1966}] \label{thm:Faulkner}
If $n \geq 2k \geq 2$, then $\binom{n}{k}$ has a prime divisor
$p \geq \frac{7}{5} k$, where the factor $\frac{7}{5}$ is best possible.
\end{theorem}

\begin{theorem}[Hanson \cite{Hanson:1973}] \label{thm:Hanson}
If $n \geq 2k \geq 2$, then $\binom{n}{k}$ has a prime divisor $p > \frac{3}{2} k$,
except for the cases $(n,k) \in \set{(4,2), (9,2), (10,5)}$.
\end{theorem}

Theorems~\ref{thm:Faulkner} and \ref{thm:Hanson} are stronger versions of
a theorem of Sylvester~\cite{Sylvester:1892}, independently discovered by
Schur~\cite{Schur:1929}, which states that if $n \geq 2k \geq 2$ then
$\binom{n}{k}$ has a prime divisor $p > k$.
A simple proof was given by Erd\H{o}s~\cite{Erdos:1934}.
Considering the special case $\binom{2n}{n}$ for $n \geq 2$ implies Bertrand's
postulate that there always exists a prime $p$ with $n < p < 2n$.
We need the following improvements.

\begin{theorem}[Nagura \cite{Nagura:1952}] \label{thm:Nagura}
If $n \geq 25$, then there exists a prime $p$ such that
$n < p < \frac{6}{5} n$.
\end{theorem}

\begin{theorem}[Dusart~\cite{Dusart:2018}] \label{thm:Dusart}
If $n \geq 89\,693$, then there exists a prime $p$ such that
$n < p < (1 + \log^{-3} n) \, n$.
\end{theorem}

\begin{lemma} \label{lem:binom-recip}
Let $n, r \geq 2$. Then there exists an odd prime $p$ with
$p \mid \binom{r+n}{n}$. In particular,
\[
  2^{r+n} \bigg/ \binom{r+n}{n} \notin \ZZ.
\]
\end{lemma}

\begin{proof}
Since ${\binom{r+n}{n} = \binom{r+n}{r}}$ by symmetry, we can assume that
$r \geq n$. Thus, we have $r+n \geq 2n$ and $n \geq 2$.
By Theorem~\ref{thm:Faulkner} there exists a prime $p \geq \tfrac{7}{5} n$,
so $p \geq 3$ that divides $\binom{r+n}{n}$.
\end{proof}

For a polynomial ${f(x) \in \QQ[x]}$, its denominator $\denom(f(x))$
is the smallest positive integer $d$ such that $d \cdot f(x) \in \ZZ[x]$,
the latter polynomial having coprime coefficients.
This definition includes the usual definition of $\denom(q)$ for $q \in \QQ$.
In particular, $\denom(q) = 1$ if and only if $q \in \ZZ$.

\begin{lemma} \label{lem:denom}
Let $n \geq 0$ and define the polynomial
\[
  f(x) = \sum_{\nu=0}^{n} a_\nu \, x^\nu
\]
with rational coefficients $a_\nu$. Then
\[
  \denom( f(x) ) = \lcm( \denom(a_0), \dots, \denom(a_n) ).
\]
For $\ell \in \ZZ$, we have
\[
  \denom( f(\ell) ) \mid \denom( f(x) ).
\]
\end{lemma}

\begin{proof}
This follows from
$\ord_p(f(x)) = \min\limits_{0 \leq \nu \leq n} \ord_p( a_\nu )$ and
$\ord_p( a_\nu \, \ell^\nu) \geq \ord_p( a_\nu )$ for any prime $p$.
\end{proof}

\begin{lemma} \label{lem:denom-2}
Let $n, r \geq 1$. For $0 \leq k \leq n$, we have
\[
  \denom \mleft( \frac{(n)_k}{(r+k)_k} \mright)
    \mids \denom \mleft( \frac{(n)_n}{(r+n)_n} \mright).
\]
If $n +r = p^e$ with $p$ a prime and $e \geq 1$, then we have for
$0 \leq k < n$ the strict inequalities
\[
  \ord_p \mleft( \frac{(r+n)_n}{n!} \mright)
    > \ord_p \mleft( \frac{(r+k)_k}{(n)_k} \mright).
\]
\end{lemma}

\begin{proof}
Let $k \in \set{0,\ldots,n}$. We then have
\[
  \frac{(r+n)_n}{n!} = \frac{(r+n)_{n-k}}{(n-k)!} \, \frac{(r+k)_k}{(n)_k}
    = \binom{r+n}{n-k} \frac{(r+k)_k}{(n)_k}.
\]
For any prime $p$, we have $\binom{r+n}{n-k} \in \ZZ_p$. This shows that
\[
  \ord_p \mleft( \frac{(r+n)_n}{n!} \mright)
    \geq \ord_p \mleft( \frac{(r+k)_k}{(n)_k} \mright),
\]
implying the first claim. Now, if ${n +r = p^e}$, then
by Lemma~\ref{lem:binom-p} we have that $\binom{r+n}{n-k} \in p\ZZ_p$
for $0 \leq k < n$, proving the second claim.
\end{proof}

\begin{lemma} \label{lem:congr-ell}
For $\ell \in \ZZ$ and $n \geq 2$, we have
\[
  \frac{\ell^{n}-(\ell-1)^{n}}{n} \notin \ZZ.
\]
\end{lemma}

\begin{proof}
The cases $\ell \in \set{0,1}$ are trivial. Define
$f_n(\ell) = \ell^{n}-(\ell-1)^{n}$.
It is easy to see that we have a reflection relation by
\[
  f_n(\ell) = (-1)^{n+1} f_n(1-\ell).
\]
Thus, there remains to consider the integers $\ell \geq 2$.
Now fix $\ell, n \geq 2$ and assume to the contrary that
\[
  \ell^{n} \equiv (\ell-1)^{n} \pmod{n}.
\]
We have $g = \gcd(n,\ell(\ell-1)) = 1$. Otherwise, $p \mid g$ would imply
a congruence of the type $0 \equiv (\pm 1)^n \pmod{p}$. Since $g = 1$ and
$2 \mid \ell(\ell-1)$, we have $2 \nmid n$. Next we choose the smallest prime
divisor $p \geq 3$ of $n$. We then obtain $b \equiv \ell/(\ell-1) \not\equiv 1 \pmod{p}$
and arrive at $b^n \equiv 1 \pmod{p}$. By Fermat's little theorem, we have
$b^e \equiv 1 \pmod{p}$ with a minimal exponent $e = (p-1)/d > 1$ and $d \mid p-1$.
As a consequence, we infer that $e \mid n$, but this contradicts the assumption
that $p$ is the smallest prime divisor of $n$.
\end{proof}

\begin{remark}
The special case $\ell = 2$ of Lemma~\ref{lem:congr-ell} was handled in \cite{Lopez-Aguayo:2015},
but without giving a proof. Actually, a proof was given in the same
issue as a solution to the initial problem of Chiri\c{t}\u{a} \cite{Chirita:2014}.
(The editors noted there that the fact that $n \nmid 2^n-1$ for $n > 1$ goes back
to a proposed problem in 1972.) A further proof of that case was also given later
in \cite{Thongjunthug:2019}.
\end{remark}

Euler's beta function is defined by
\[
  \BF(x,y) = \int_{0}^{1} t^{x-1} (1-t)^{y-1} dt
\]
for $\re(x) > 0$ and $\re(y) > 0$, which satisfies the identity
\[
  \BF(x,y) = \frac{\Gamma(x) \Gamma(y)}{\Gamma(x+y)},
    \quad \text{where $\Gamma$ is the gamma function.}
\]


\section{Proofs of basic theorems}
\label{sec:proofs-1}

\begin{lemma} \label{lem:S-int}
Let $n, r \geq 1$ and $x \in \CC$. Then
\begin{align}
  \MS_{(r,n)}(x) &= r\! \int_{0}^{1} (x-t)^n (1-t)^{r-1} dt.
  \label{eq:S-int} \\
\shortintertext{In particular, we have}
  \MS_{(1,n)}(x) &= \frac{x^{n+1}-(x-1)^{n+1}}{n+1},
  \label{eq:S-1n} \\
  \MS_{(r,1)}(x) &= x - \frac{1}{r+1},
  \label{eq:S-r1} \\
\shortintertext{and}
  \MS_{(r,n)}(1) &= \frac{r}{r+n}.
  \label{eq:S-1}
\end{align}
\end{lemma}

\begin{proof}
Using the beta function, we infer that
\[
  \int_{0}^{1} (x-t)^n (1-t)^{r-1} dt
    = \sum_{k=0}^{n} \binom{n}{k} (-1)^k \, x^{n-k} \, \BF(k+1,r).
\]
Since $r \, \BF(k+1,r) = 1 / \binom{r+k}{r}$,
this establishes \eqref{eq:S-int} using \eqref{eq:S-def}.
Direct evaluations provide that
\begin{align*}
  \MS_{(1,n)}(x) &= \int_{0}^{1} (x-t)^n dt
    = - \left. \frac{(x-t)^{n+1}}{n+1} \, \right|_{0}^{1}
    = \frac{x^{n+1}-(x-1)^{n+1}}{n+1} \\
\shortintertext{and}
  \MS_{(r,n)}(1) &= r\! \int_{0}^{1} (1-t)^{r+n-1} dt
    = - r \left. \frac{(1-t)^{r+n}}{r+n} \, \right|_{0}^{1}
    = \frac{r}{r+n}.
\end{align*}
Formula \eqref{eq:S-r1} is given by \eqref{eq:S-def} with $n=1$.
\end{proof}

\begin{proof}[Proof of Theorem~\ref{thm:basic}]
We have to show six parts.

(i). This follows from \eqref{eq:S-def} and using $\binom{n}{k} = \frac{(n)_k}{k!}$
and $\binom{r+k}{r} = \binom{r+k}{k} = \frac{(r+k)_k}{k!}$.

(ii). This is given by Lemma~\ref{lem:S-int} and \eqref{eq:S-int}.

(iii), (iv). Differentiating \eqref{eq:S-int} with respect to $x$ yields
$\MS_{(r,n)}(x)' = n \, \MS_{(r,n-1)}(x)$. Together with $\MS_{(r,0)}(x) = 1$
by \eqref{eq:S-0}, the polynomials $\MS_{(r,n)}(x)$ for $n \geq 0$ form an Appell sequence.
As a consequence, part (iv) is equivalent to part (iii), see \cite{Appell:1880}.

(v), (vi). We use part (i) and apply Lemmas~\ref{lem:denom} and \ref{lem:denom-2}.
This shows part (v). Part (vi) follows from using Lemma~\ref{lem:denom} again.
This proves the theorem.
\end{proof}

\begin{proof}[Proof of Corollary~\ref{cor:basic}]
The results are derived from Theorem~\ref{thm:basic}~(ii) and (iv).
We show the claims in order of their dependencies.

(iii). Evaluating the integral \eqref{eq:S-int}, Lemma~\ref{lem:S-int} gives
the result by \eqref{eq:S-1}.

(i). It then follows that
\[
  \MS_{(r,n)}(x+1) = \sum_{k=0}^{n} \binom{n}{k} x^{n-k} \MS_{(r,k)}(1).
\]

(ii), (iv). Note that $\MS_{(r,n)}(0) = (-1)^n \, \binom{r+n}{r}^{\!\!-1}$
by \eqref{eq:S-def}. The identities follow from taking $x = \pm 1$.

(v). We consider the integrand of \eqref{eq:S-int}. For $t \in (0,1)$,
we have $(1-t)^{r-1} > 0$, as well as $(x-t)^n > 0$ for $x \geq 1$
and $(-1)^n (x-t)^n > 0$ for $x \leq 0$.
\end{proof}

\begin{remark}
The partial fraction decomposition
\[
  \sum_{k=0}^{n} \binom{n}{k} \frac{(-1)^k}{x+k}
    = \frac{1}{x} \binom{x+n}{n}^{\!\!-1}
\]
and its inversion
\[
  \frac{x}{x+n} = \sum_{k=0}^{n} \binom{n}{k} (-1)^k \binom{x+k}{k}^{\!\!-1}
\]
are well known, cf.~\cite[Sec.~5.3, p.~196]{GKP:1994} and
\cite[\S\,4, p.~54]{Norlund:1924}. Instead of using finite differences,
the identities are derived here from the integral \eqref{eq:S-int} and
the Appell properties of $\MS_{(r,n)}(x)$.
\end{remark}

\begin{proof}[Proof of Theorem~\ref{thm:row}]
Let $n, r \geq 1$. We have to show four parts.

(i), (ii). It is easy to verify that
\[
  \binom{n}{k} \binom{r+n-k}{r}^{\!\!-1}
    = \binom{r+n}{k} \binom{r+n}{n}^{\!\!-1}.
\]
We reverse the summation of \eqref{eq:S-def} and use the above identity. Thus,
\[
  \MS_{(r,n)}(x)
    = \sum_{k=0}^{n} \binom{n}{k} (-1)^{n-k} \, x^k \binom{r+n-k}{r}^{\!\!-1}
    \!\!\!= \binom{r+n}{n}^{\!\!-1} \sum_{k=0}^{n} \binom{r+n}{k} (-1)^{n-k} x^k\!.
\]
By taking $x=-1$, the formula for $\MS_{(r,n)}(-1)$ follows.

(iii). Due to the symmetry of the binomial coefficients, we sum from the left-hand
and right-hand side in a row of Pascal's triangle. Therefore, this yields
\begin{equation} \label{eq:S-recip-1}
  \sum_{k=0}^{n} \binom{r+n}{k} + \sum_{k=0}^{r} \binom{r+n}{k}
    = 2^{r+n} + \binom{r+n}{n},
\end{equation}
where $\binom{r+n}{n} = \binom{r+n}{r}$ is counted twice.
Considering the sign and the extra factor of $\MS_{(r,n)}(-1)$,
we finally obtain the reciprocity relation
\begin{equation} \label{eq:S-recip-2}
  (-1)^n \MS_{(r,n)}(-1) + (-1)^r \MS_{(n,r)}(-1)
    = 2^{r+n} \binom{r+n}{n}^{\!\!-1} + 1.
\end{equation}

(iv). First we have $\MS_{(r,1)}(-1) \notin \ZZ$ by \eqref{eq:S-r1}.
Further we infer from applying Lemmas~\ref{lem:congr-ell} and \ref{lem:S-int} that
$\MS_{(1,n)}(-1) \notin \ZZ$. Thus, we can now assume that $n,r \geq 2$.
Lemma~\ref{lem:binom-recip} shows that the right-hand side of \eqref{eq:S-recip-2}
is not integral, implying that $\MS_{(r,n)}(-1)$ and $\MS_{(n,r)}(-1)$ cannot be
both integers. As a consequence, ${\MS_{(n,n)}(-1) \notin \ZZ}$ for $n \geq 1$.
Further, direct computations via \eqref{eq:S-recip-1} and \eqref{eq:S-recip-2}
imply that
\begin{align*}
  (-1)^n \MS_{(n+1,n)}(-1) &= 2^{2n} \binom{2n+1}{n}^{\!\!-1} \notin \ZZ \\
\shortintertext{and}
  (-1)^{n+1} \MS_{(n,n+1)}(-1) &= 2^{2n} \binom{2n+1}{n}^{\!\!-1} + 1 \notin \ZZ,
\end{align*}
using the same arguments from above. This proves the theorem.
\end{proof}

\begin{proof}[Proof of Corollary~\ref{cor:row}]
The first equation follows from combining Theorem~\ref{thm:basic}~(ii) and
Theorem~\ref{thm:row}~(ii), the second one from Corollary~\ref{cor:basic}~(i)
with $x = -2$.
\end{proof}

\begin{proof}[Proof of Theorem~\ref{thm:recip}]
Let $n, r \geq 1$. We introduce the notation
\[
  (x+y)^{n,m} = \sum_{k=0}^{m} \binom{n}{k} x^{n-k} y^k \quad (0 \leq m \leq n)
\]
for partial sums of the binomial identity, which is not commutative in general.
It is easy to see that
\begin{equation} \label{eq:part-sum}
  (x+y)^{n,m} + (y+x)^{n,n-m} = (x+y)^n + \binom{n}{m} x^{n-m} y^m.
\end{equation}

From Theorem~\ref{thm:row}~(i), it then follows that
\begin{align*}
  \MS_{(r,n)}(x) &= (-1)^r \binom{r+n}{r}^{\!\!-1} (-1+x)^{r+n,n}, \\
  x^r \MSR_{(r,n)}(x) &= (-1)^n \binom{r+n}{r}^{\!\!-1} (x-1)^{r+n,n}.
\end{align*}
Using \eqref{eq:part-sum}, we finally derive that
\[
  \MS_{(r,n)}(x) + x^n \MSR_{(n,r)}(x)
    = (-1)^r (x-1)^{r+n} \binom{r+n}{r}^{\!\!-1} + x^n. \qedhere
\]
\end{proof}

\begin{proof}[Proof of Proposition~\ref{prop:recur}]
Let $n, r \geq 1$.  We have to show four parts, where we make use of the
integral formula \eqref{eq:S-int}.

(i). Rewriting the integrand by
\begin{align*}
  (x-t)^n (1-t)^{r-1} &= (x-1+1-t) (x-t)^{n-1} (1-t)^{r-1} \\
    &= (x-1) (x-t)^{n-1} (1-t)^{r-1} + (x-t)^{n-1} (1-t)^r
\end{align*}
implies the formula
\begin{equation} \label{eq:recur-1}
  \MS_{(r,n)}(x)
    = (x-1) \, \MS_{(r,n-1)}(x) + \frac{r}{r+1} \, \MS_{(r+1,n-1)}(x).
\end{equation}

(ii). We use integration by parts that
\[
  f(t)g(t) \bigg|_{0}^{1}
    = \int_{0}^{1} f(t)g'(t) dt + \int_{0}^{1} f'(t)g(t) dt.
\]
Set $f(t) = - (x-t)^n$ and $g(t) = (1-t)^r$. Then we obtain the equation
\begin{equation} \label{eq:recur-2}
  x^n = \MS_{(r,n)}(x) + \frac{n}{r+1} \, \MS_{(r+1,n-1)}(x).
\end{equation}

(iii). Subtracting \eqref{eq:recur-1} from \eqref{eq:recur-2} and shifting
the index by $n \mapsto n+1$ yield
\[
  \MS_{(r+1,n)}(x)
    = \frac{r+1}{r+n+1} \left( x^{n+1} - (x-1) \MS_{(r,n)}(x) \right).
\]

(iv). Multiply \eqref{eq:recur-1} by $n$ and \eqref{eq:recur-2} by $r$,
respectively, and subtract the equations. Divide the resulting equation by
$n+r$ and shift the index by $n \mapsto n+1$. Finally, this gives the equation
\[
  \MS_{(r,n+1)}(x)
    = \frac{r}{r+n+1} x^{n+1} + \frac{n+1}{r+n+1} (x-1) \MS_{(r,n)}(x),
\]
completing the proof.
\end{proof}

For $n, r \geq 1$, recall by \eqref{eq:psi-1} the function
\begin{equation} \label{eq:psi-def}
  \psi_{(r,n)}(x) = \sum_{k=0}^{r-1} \binom{n+k}{k} (-1)^k (x-1)^{r-1-k},
\end{equation}
which satisfies the recurrence
\begin{equation} \label{eq:psi-rec}
  \psi_{(r+1,n)}(x) = (x-1) \psi_{(r,n)}(x) + (-1)^r \binom{n+r}{r}.
\end{equation}

\begin{lemma} \label{lem:psi-congr}
Let $n, r \geq 1$ and $\ell \in \ZZ$. For $d \in \set{1,\ldots,r}$, we have
\[
  r! \, \psi_{(r,n)}(\ell)
    \equiv r! \, \ell^{d-1} (\ell-1)^{r-d} \pmod{n+d}.
\]
\end{lemma}

\begin{proof}
Let $n, r \geq 1$ and $1 \leq d \leq r$. We infer for $k \geq 0$ that
\[
  (n+k)_k (-1)^k \equiv (d-1)_k \pmod{n+d}.
\]
Using \eqref{eq:psi-def}, we obtain for $\ell \in \ZZ$ that
\begin{align*}
  r! \, \psi_{(r,n)}(\ell)
    &\equiv \sum_{k=0}^{r-1} \frac{r!}{k!} (n+k)_k (-1)^k (\ell-1)^{r-1-k} \\
    &\equiv \sum_{k=0}^{r-1} \frac{r!}{k!} (d-1)_k (\ell-1)^{r-1-k} \\
    &\equiv r! \, (\ell-1)^{r-d} \sum_{k=0}^{d-1} \binom{d-1}{k} (\ell-1)^{d-1-k} \\
    &\equiv r! \, \ell^{d-1} (\ell-1)^{r-d} \pmod{n+d},
\end{align*}
as desired.
\end{proof}

\begin{proof}[Proof of Theorem~\ref{thm:ident}]
Let $n, r \geq 1$. We first show that
\begin{equation} \label{eq:S-psi}
  \MS_{(r,n)}(x) = (-1)^{r} \binom{n+r}{r}^{\!-1} \!
    \left( (x-1)^{n+r} - x^{n+1} \psi_{(r,n)}(x) \right).
\end{equation}
Now, fix $n$. We use induction on $r$.
By Lemma~\ref{lem:S-int} and~\eqref{eq:S-1n}, we have
\[
  \MS_{(1,n)}(x) = \frac{x^{n+1}-(x-1)^{n+1}}{n+1},
\]
which coincides with \eqref{eq:S-psi} in the case $r=1$. Inductive step:
we assume that \eqref{eq:S-psi} holds for $r$ and prove for $r+1$.
We use the recurrence formula of Proposition~\ref{prop:recur}~(iii).
Thus, we obtain
\begin{align*}
 \frac{r+n+1}{r+1} \MS_{(r+1,n)}(x)
 &= x^{n+1} - (x-1) \MS_{(r,n)}(x) \\
 &= (-1)^{r} \binom{n+r}{r}^{\!-1}
   \bigg( (-1)^r \binom{n+r}{r} x^{n+1} - (x-1)^{n+r+1} \\
 &\quad + x^{n+1} (x-1) \psi_{(r,n)}(x) \bigg) \\
 &= (-1)^{r+1} \binom{n+r}{r}^{\!-1}
   \left( (x-1)^{n+r+1} - x^{n+1} \psi_{(r+1,n)}(x) \right)\!,
\end{align*}
where the last equation follows from \eqref{eq:psi-rec}.

This implies that $\MS_{(r+1,n)}(x)$ is equal to~\eqref{eq:S-psi} in the case
$r+1$. Finally, identities \eqref{eq:S-ident} and \eqref{eq:psi-1} are
established. To show the alternative identity \eqref{eq:psi-2}, we have by
Theorem~\ref{thm:recip} that
\[
  \MS_{(r,n)}(x) + x^n \MSR_{(n,r)}(x)
    = (-1)^r (x-1)^{r+n} \binom{r+n}{r}^{\!\!-1} + x^n.
\]
By Theorem~\ref{thm:row}~(i), we can write
\begin{align*}
  x^n \MSR_{(n,r)}(x) &= (-1)^r x^n \binom{r+n}{r}^{\!\!-1}
    \sum_{k=0}^{r} \binom{r+n}{k} (-1)^k \, x^{r-k} \\
  &= x^n + (-1)^r x^{n+1} \, \binom{r+n}{r}^{\!\!-1} \widetilde{\psi}_{(r,n)}(x),
\end{align*}
where
\begin{equation} \label{eq:S-psi-tilde}
  \widetilde{\psi}_{(r,n)}(x)
    = \sum_{k=0}^{r-1} \binom{r+n}{k} (-1)^k x^{r-1-k}.
\end{equation}
Putting all together shows \eqref{eq:S-ident}, but holding with \eqref{eq:psi-2}.
Consequently, we obtain
\[
  \psi_{(r,n)}(x) = \widetilde{\psi}_{(r,n)}(x).
\]

Lastly, parts (i) and (ii) are given by \eqref{eq:S-1n} and \eqref{eq:S-r1},
respectively. Part~(iii) follows from \eqref{eq:S-psi} by taking $x=2$.
This proves the theorem.
\end{proof}

\begin{remark}
Searching for identities similar to \eqref{eq:S-psi} in the literature,
one finds the following identity in Gould's tables of combinatorial identities
of 1972 (see~\cite[Eq.~(4.13), p.~47]{Gould:1972}) that
\[
  \sum_{k=0}^{n} \binom{n}{k} \frac{x^k}{\binom{k+r}{k}}
    = 1 + \frac{(x+1)^{n+r} - \sum_{k=0}^{r} \binom{n+r}{k}x^k}{x^r \binom{n+r}{n}},
\]
from which one can also deduce formula \eqref{eq:S-psi} with \eqref{eq:S-psi-tilde}.
\end{remark}


\section{Proofs of main theorems}
\label{sec:proofs-2}

\begin{proof}[Proof of Proposition~\ref{prop:congr-1}]
This easily follows from Theorem~\ref{thm:row}~(i).
\end{proof}

\begin{proof}[Proof of Proposition~\ref{prop:congr-2}]
Let $n, r \geq 1$ and $\ell \in \ZZ \setminus \set{0,1}$.
We use Theorem~\ref{thm:ident} and assume that $\MS_{(r,n)}(\ell) \in \ZZ$.
Then the numerator of \eqref{eq:S-ident} must be divisible by each factor
$n+d$ of the denominator for $1 \leq d \leq r$. Using Lemma~\ref{lem:psi-congr},
we infer the necessary but not sufficient conditions that
\begin{align*}
  0 &\equiv r! \left( (\ell-1)^{r+n} - \ell^{n+1} \, \psi_{(r,n)}(\ell) \right) \\
    &\equiv r! \, (\ell-1)^{r-d} \left( (\ell-1)^{n+d} - \ell^{n+d} \right)
    \equiv a_d \pmod{n+d}
\end{align*}
for $1 \leq d \leq r$. Conversely, if one $a_d \not\equiv 0 \pmod{n+d}$,
then $\MS_{(r,n)}(\ell) \notin \ZZ$.
\end{proof}

\begin{proof}[Proof of Theorem~\ref{thm:main}]
Let $n, r \geq 1$ and $\ell \in \ZZ$. We show six conditions in order of
their dependencies that imply $\MS_{(r,n)}(\ell) \notin \ZZ$ and also
$\MS_{(n,r)}(\ell) \notin \ZZ$ by symmetry of $\binom{r+n}{r} = \binom{n+r}{n}$.
It is easy to see that the results also hold for $-\ell$ except for part (v).

(i). This follows from Lemmas~\ref{lem:congr-ell} and \ref{lem:S-int}.

(iii). By Theorem~\ref{thm:basic} we have
\[
  \MS_{(r,n)}(\ell)
    = \sum_{k=0}^{n} (-1)^k \, \ell^{n-k} \, \frac{(n)_k}{(r+k)_k}.
\]
Since $r+n=p^e$ is a prime power with $e \geq 1$,
Lemma~\ref{lem:denom-2} shows that the last summand satisfies that
\[
  v = \ord_p \mleft( \frac{n!}{(r+n)_n} \mright)
    < \ord_p \mleft( \frac{(n)_k}{(r+k)_k} \mright)
    \leq \ord_p \mleft( \ell^{n-k} \frac{(n)_k}{(r+k)_k} \mright)
\]
for $0 \leq k < n$. By Lemma~\ref{lem:binom-p}, we have
$(r+n)_n/n! = \binom{r+n}{n} \in p\ZZ_p$, so $v < 0$.
Using Lemma~\ref{lem:val-min}, the result follows.

(v). This is given by Corollary~\ref{cor:basic}~(ii) and (iii) for
$\ell \in \set{0,1}$, and by Theorem~\ref{thm:row}~(iv) for $\ell = -1$.

(vi). Let $b = \binom{r+n}{r}$ and $g = \gcd (b, \ell) > 1$.
Note that case $\ell = 0$ is compatible with Corollary~\ref{cor:basic}~(ii),
since $g = b$. So we assume that $|\ell| \geq 2$. Now fix a prime divisor
$p$ of $g$. We use Theorem~\ref{thm:ident} to derive that
\[
  (-1)^{r} \MS_{(r,n)}(\ell) = \frac{(\ell-1)^{r+n}}{b}
    - \frac{\ell^{n+1} \, \psi_{(r,n)}(\ell)}{b} = f_1 - f_2.
\]
Assume that $\psi_{(r,n)}(\ell) \neq 0$; otherwise, we are done.
Since $p \mid g$ and $\gcd(\ell-1,\ell) = 1$, we obtain for the fractions that
\[
  - \ord_p(b) = \ord_p(f_1) < \ord_p(f_2).
\]
From Lemma~\ref{lem:val-min}, we finally infer that
\[
  \ord_p(\denom(\MS_{(r,n)}(\ell))) = \ord_p( b ) > 0.
\]

(ii). The case $r=n$ is postponed and borrowed from the proof of
Theorem~\ref{thm:diag} below, which uses the independent part (vi).

(iv). Assume that $r$ and $n$ are odd, and $\ell$ is even. The case $\ell = 0$
is handled by part~(v), so $|\ell| \geq 2$. Let $b = \binom{r+n}{r}$ and
$g = \gcd (b, \ell)$. Since
\[
  b \equiv \binom{r+n-1}{r-1} \frac{r+n}{r} \equiv 0 \pmod{2},
\]
we have $2 \mid g$, and we can apply part (vi).
This completes the proof.
\end{proof}

\begin{proof}[Proof of Theorem~\ref{thm:diag}]
The first formula is given by Theorem~\ref{thm:row}~(i), where
\[
  \MS_{(1,1)}(x) = x - \tfrac{1}{2}.
\]
The recurrence formula follows
from Proposition~\ref{prop:recur}~(iii) and (iv). Hence,
\begin{equation} \label{eq:S-recur}
  \MS_{(n+1,n+1)}(x) = A_n(x) + \tfrac{1}{2} x^{n+1}
\end{equation}
with
\[
  A_n(x) = \frac{n+1}{2n+1} \frac{x-1}{2}
    \left( x^{n+1} - (x-1) \MS_{(n,n)}(x) \right).
\]

Now, let $\ell \in \ZZ$ be odd. We use Lemma~\ref{lem:val-min} implicitly.
We have $\MS_{(1,1)}(\ell) \in \frac{1}{2} \ZZ_2 \setminus \ZZ_2$.
We use induction on $n$. Assume that
$\MS_{(n,n)}(\ell) \in \frac{1}{2} \ZZ_2 \setminus \ZZ_2$.
Since $\ell-1$ is even, it follows that $ A_n(\ell) \in \ZZ_2$.
From \eqref{eq:S-recur} we deduce that
${\MS_{(n+1,n+1)}(\ell) \in \frac{1}{2} \ZZ_2 \setminus \ZZ_2}$.
Finally, we obtain that $\ord_2(\MS_{(n,n)}(\ell)) = -1$ for all $n \geq 1$.

In the other case, where $\ell \in \ZZ$ is even, we use
Theorem~\ref{thm:main}~(vi). For $n \geq 1$, we have
$\binom{2n}{n} = 2 \binom{2n-1}{n-1}$,
and thus $2 \mid g = \gcd \mleft( \binom{2n}{n}, \ell \mright)$.
Using \eqref{eq:binom-2}, we then infer that
\[
  \ord_2(\MS_{(n,n)}(\ell)) = -\ord_2 \mleft( \binom{2n}{n} \mright)
    = - (2s_2(n) - s_2(2n)) = - s_2(n).
\]

As a result, $\MS_{(n,n)}(\ell) \notin \ZZ$ for $n \geq 1$ and $\ell \in \ZZ$.
This proves the theorem.
\end{proof}

\begin{remark}
Theorem~\ref{thm:diag} implies for $n \geq 1$ and odd $\ell \in \ZZ$ that
\[
  \MS_{(n,n)}(\ell) - \MS_{(n,n)}(1) = \binom{2n}{n}^{\!\!-1}
    \sum_{k=0}^{n} \binom{2n}{k} (-1)^{n-k} \, (\ell^k - 1) \in \ZZ_2.
\]
However, a direct proof via $2$-adic valuation of the above summands seems to
be complicated. More generally, it follows for odd $\ell, \ell_2 \in \ZZ$ that
\[
  \MS_{(n,n)}(\ell) - \MS_{(n,n)}(\ell_2) \in \ZZ_2.
\]
\end{remark}

\begin{proof}[Proof of Corollary~\ref{cor:density}]
We first consider Theorem~\ref{thm:Singmaster}. By a simple counting argument
and excluding those binomial coefficients that equal $1$, we arrive at an
equivalent formulation as follows. For $d \geq 1$ and $m \geq 2$ define the sets
\[
  \widetilde{B}_m(d) = \set{ (r,n) \in \ZZ^2 : n,r \geq 1, \; 1 \leq r+n \leq m,
    \,\text{and}\;\; d \mid \binom{r+n}{r} }.
\]
Then we also have the density
\[
  \lim_{m \to \infty} \# \widetilde{B}_m(d) \Big/ \binom{m}{2} = 1.
\]
Let $d = |\ell| \geq 2$. By Theorem~\ref{thm:main}~(vi), we have that
\[
  \gcd \mleft( \binom{r+n}{r}, d \mright)
    > 1 \impliesq \MS_{(r,n)}(\ell) \notin \ZZ.
\]
The stronger condition also shows that
\[
  d \mid \binom{r+n}{r} \impliesq \MS_{(r,n)}(\ell) \notin \ZZ.
\]
Therefore, we infer that $\# \NI_m(\ell) \geq \# \widetilde{B}_m(d)$,
implying the result.
\end{proof}

\begin{proof}[Proof of Theorem~\ref{thm:range}]
Let $n, r \geq 2$ and $\ell \in \ZZ \setminus \set{0,1}$.
We have to show two parts.

(i). Assume that $n > r \geq \frac{1}{5} n$ and $n \geq |\ell-1|$.
First we consider the case $n \geq 25$. Using Theorem~\ref{thm:Nagura},
we infer that there exists a prime~$p$ satisfying $n < p < n+r$,
since $n+r \geq (1 + \frac{1}{5}) n$. By assumption we have
\[
  r! \, (\ell-1) \not\equiv 0 \pmod{p} \andq
    (\ell-1)^p - \ell^p \equiv -1 \pmod{p}.
\]
Hence, applying Proposition~\ref{prop:congr-2} yields that
$\MS_{(r,n)}(\ell) \notin \ZZ$ for $n \geq 25$.
Checking Table~\ref{tbl:except-1} reveals that the result also holds for
the remaining case where $1 < n < 25$.
As a refinement, Theorem~\ref{thm:Dusart} allows us to take the condition
$r > n / \!\log^3 n$ for $n \geq 89\,693$.

(ii). Assume that $r > n$. We apply Theorem~\ref{thm:Hanson} to
$\binom{r+n}{n}$, where the exceptions
\[
  \binom{r+n}{n} \in \set{\binom{4}{2},\binom{9}{2},\binom{10}{5}}
\]
are ruled out by the excluded condition $r=n$, and by
Theorem~\ref{thm:main}~(iii) that $r+n$ is a prime power.
Therefore, we can continue without restrictions. Then there exists a prime
$p > \frac{3}{2} n$ that divides exactly one of the numbers
$r+1, \ldots, r+n$, say $r+d$ with $d \in \set{1,\ldots,n}$.
This also implies that $p \nmid r$. We split the proof into two cases as follows.

\Case{$p \nmid \ell - 1$}
By Corollary~\ref{cor:basic}~(i), we have that
\begin{equation} \label{eq:S-frac}
  \MS_{(r,n)}(\ell) = \sum_{k=0}^{n} \binom{n}{k} (\ell-1)^{n-k} \frac{r}{r+k}.
\end{equation}
All summands of \eqref{eq:S-frac} are $p$-integral, except for $k = d$ where
\begin{equation} \label{eq:S-frac-p}
  \ord_p \left( \binom{n}{d} (\ell-1)^{n-d} \, r\right) = 0 \andq
    \ord_p \left( \frac{1}{r+d} \right) < 0.
\end{equation}
Thus, Lemma~\ref{lem:val-min} implies that $\ord_p(\MS_{(r,n)}(\ell)) < 0$
and finally $\MS_{(r,n)}(\ell) \notin \ZZ$.

\Case{$p \mid \ell - 1$}
We infer from \eqref{eq:S-frac} and \eqref{eq:S-frac-p} that if
\[
  (n-d) \ord_p (\ell-1) < \ord_p (r+d)
\]
(being always true for $d=n$), then $\MS_{(r,n)}(\ell) \notin \ZZ$.
Conversely, if and only if
\[
  d \neq n \andq (n-d) \ord_p (\ell-1) \geq \ord_p (r+d),
\]
then $\MS_{(r,n)}(\ell) \in \ZZ_p$.

To prevent the latter case $\MS_{(r,n)}(\ell) \in \ZZ_p$, which can happen
for sufficiently large~$\ell$, we have to require that $p \nmid \ell - 1$.
A priori, this is satisfied if $p > |\ell-1|$, which is handled by the
condition $n \geq \frac{2}{3} |\ell-1|$. Furthermore, the condition
$\gcd \mleft( \binom{r+n}{r}, \ell-1 \mright)=1$ also ensures that
$p \nmid \ell - 1$, but it may exclude the allowed cases in which a prime
$q < p$ satisfies $q \mid \ell - 1$ and $q \mid \binom{r+n}{r}$.
Thus, an improved condition, involving such primes, defines the set
\[
  \MP = \set{p : p > \tfrac{3}{2} n,\, p \mid \binom{r+n}{r},\, p \nmid \ell-1},
\]
which has to be nonempty. This completes the proof of the theorem.
\end{proof}

\begin{proof}[Proof of Corollary~\ref{cor:range}]
Theorem~\ref{thm:main}~(i) and (ii) cover the cases $r=1$, $n=1$, and $r=n$.
Let $n,r \geq 2$, $r \neq n$, and $\ell \in \LS$.
We consider the two parts of the proof of Theorem~\ref{thm:range}.
In both cases there exists a prime $p \geq 5 > |\ell - 1|$,
from which the result then follows.
(i). We have $p > n > r \geq 2$.
(ii). We have $p > \frac{3}{2} n \geq 3$.
\end{proof}


\section{Exceptions}
\label{sec:except}

Let $n, r \geq 1$ and $\ell \in \ZZ$. The necessary and sufficient condition
for exceptional cases, where $\MS_{(r,n)}(\ell) \in \ZZ$, can be reformulated
by Proposition~\ref{prop:congr-1} as a congruence of an incomplete binomial sum
such that
\begin{equation} \label{eq:S-congr-except}
  \sum_{k=0}^{n-1} \binom{r+n}{k} (-\ell)^k \equiv 0 \pmod{\binom{r+n}{r}}.
\end{equation}

\begin{theorem}
Let $n, r \geq 1$ and $\ell \in \ZZ$. If $\MS_{(r,n)}(\ell) \in \ZZ$,
then there exist positive integers $a$ and $b$ such that
\[
  \MS_{(r,n)}(a + b \lambda) \in \ZZ \quad (\lambda \in \ZZ),
\]
where $b = \binom{r+n}{r}$ and $1 < a < b$ with $a \equiv \ell \pmod{b}$.
As a consequence,
\[
  \MS_{(r,n)}(\ell) \notin \ZZ \quad (1 < \ell < b) \impliesq
  \MS_{(r,n)}(\ell) \notin \ZZ \quad (\ell \in \ZZ).
\]
\end{theorem}

\begin{proof}
Assume that $\MS_{(r,n)}(\ell) \in \ZZ$. Let $b = \binom{r+n}{r}$.
By Proposition~\ref{prop:congr-1}, congruence \eqref{eq:S-congr-except} holds
for $\ell$, so also for the values
\begin{equation} \label{eq:S-congr-ell}
  \ell = a + b \lambda \quad (\lambda \in \ZZ)
\end{equation}
with some integer $a \equiv \ell \pmod{b}$, where $0 \leq a < b$.
By Theorem~\ref{thm:main}~(v), the case $a \in \set{0,1}$ cannot occur,
so we have that $1 < a < b$.
Conversely, if $\MS_{(r,n)}(\ell) \notin \ZZ$ for $1 < \ell < b$,
then from \eqref{eq:S-congr-except} and \eqref{eq:S-congr-ell},
it follows that $\MS_{(r,n)}(\ell) \notin \ZZ$ for all $\ell \in \ZZ$.
\end{proof}

Extending the computations of Table~\ref{tbl:except-1} for the case $r = 2$
shows that different values of $a$ can occur for a given modulus~$b$.
\vspace*{8pt}

\begin{table}[H]
\setstretch{1.25}
\begin{center}
\begin{tabular}{*{4}{c}}
  \toprule
  \multicolumn{4}{c}{Parameters $(r,n,a,b)$} \\
  \midrule
  $(2,4,11,15)$ & $(2,8,28,45)$ & $(2,12,53,91)$ & $(2,16,35,153)$ \\
  $(2,16,86,153)$ & $(2,16,137,153)$ & $(2,20,127,231)$ & $(2,20,160,231)$ \\
  $(2,20,226,231)$ & $(2,24,176,325)$ & $(2,28,233,435)$ & $(2,32,298,561)$ \\
  \bottomrule
\end{tabular}

\caption{\parbox[t]{23em}{Exceptions where $\MS_{(r,n)}(a) \in \ZZ$ for $r=2$,
$1 \leq n \leq 32$, and $1 < a < b = \binom{r+n}{r}$.}}
\end{center}
\end{table}
\vspace*{-2pt}

The case $r=3$ shows a different and more complex pattern.

\begin{result}
Let ${1 \leq n < 200}$. The exceptions $\MS_{(3,n)}(\ell) \in \ZZ$,
for some suitable $\ell \in \ZZ$, occur for
\begin{align*}
  n \in \{& 7,18,23,31,36,39,54,55,71,87,90,95,103,108,119, \\
    & 126,127,135,144,151,159,167,180,183,198,199 \}.
\end{align*}
See Figure~\ref{fig:except-1}.
Any element $n$ of the above sequence has the property that $3+n$ has
at least two different prime factors by Theorem~\ref{thm:main}~(iii).
Checking the exceptions $\MS_{(3,18)}(\ell) \in \ZZ$ for
$1 < \ell < \binom{21}{3} = 1330$ provides the values
\[
  \ell \in \set{153,191,419,457,723,951,989,1217},
\]
as displayed in Figure~\ref{fig:except-2}.
\end{result}

\vspace*{5ex}

\begin{figure}[H]
\begin{center}
\includegraphics[width=9cm]{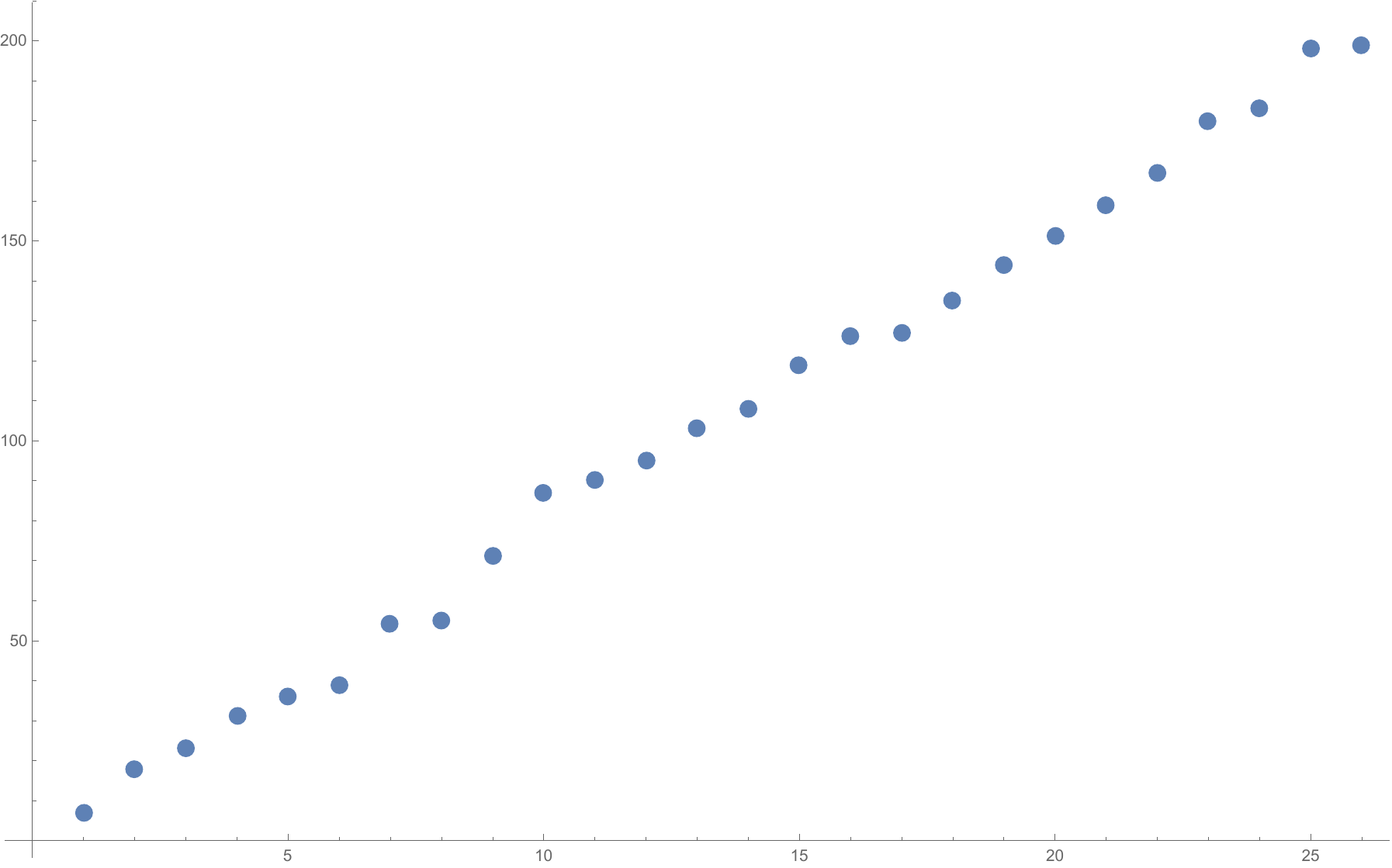}

\caption{\parbox[t]{27.1em}{Exceptions where $\MS_{(3,n)}(\ell) \in \ZZ$ for
$1 \leq n < 200$ and suitable $\ell$. Displayed values of $n$.}}
\label{fig:except-1}
\end{center}
\end{figure}

\begin{figure}[H]
\begin{center}
\includegraphics[width=9cm]{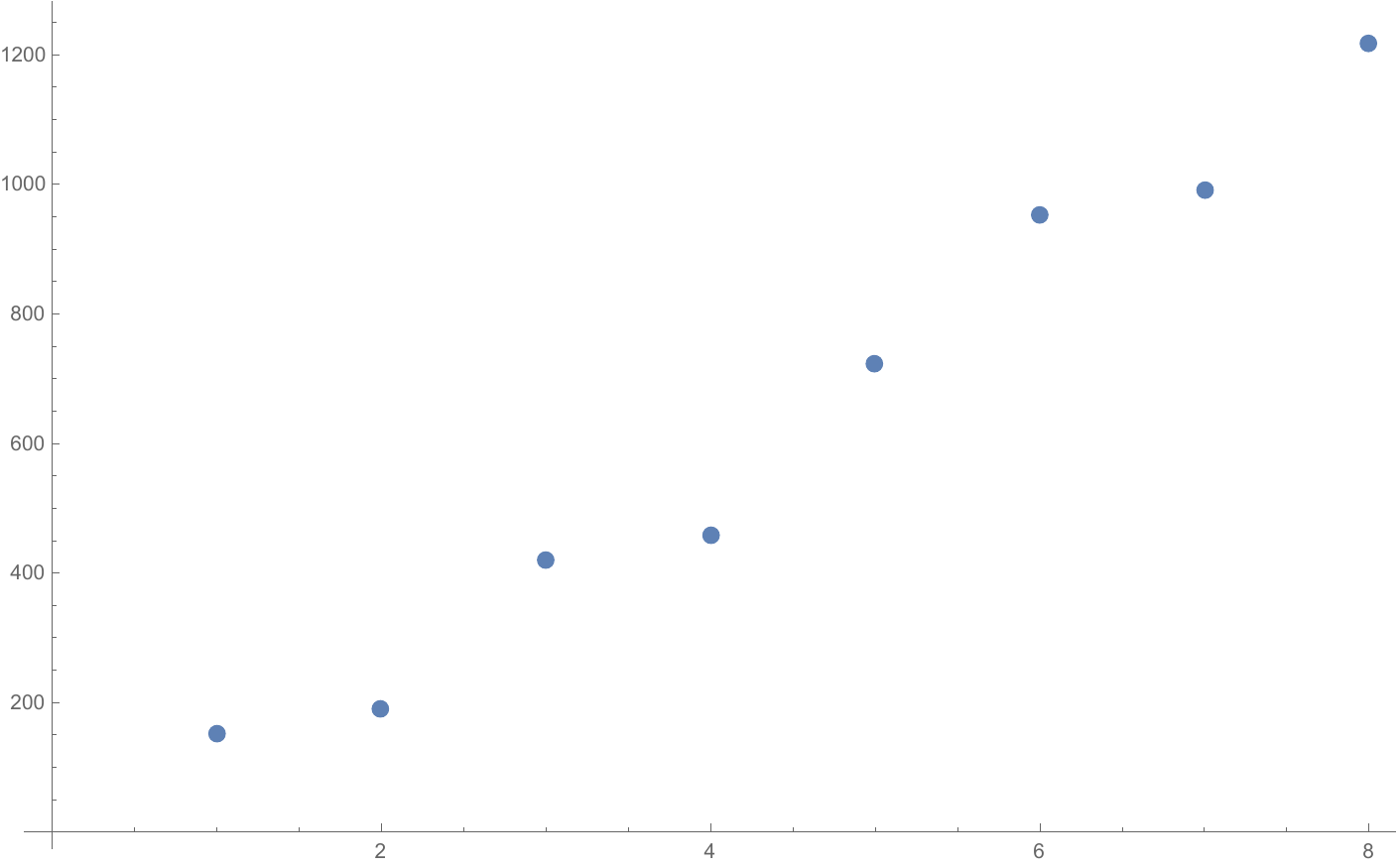}

\caption{\parbox[t]{21.9em}{Exceptions where $\MS_{(3,18)}(\ell) \in \ZZ$ for
$1 < \ell < 1330$. Displayed values of $\ell$.}}
\label{fig:except-2}
\end{center}
\end{figure}


\bibliographystyle{amsplain}

\end{document}